\documentclass[12pt]{article}
\PassOptionsToPackage{hyphens}{url}
\usepackage[hyphens,spaces,obeyspaces]{url}
\usepackage{amsmath,amssymb,amsthm,amsfonts,latexsym,bbm,xspace,graphicx,float,mathtools,mathdots,cancel}
\usepackage{braket,caption,subcaption,ellipsis,xcolor,textcomp,MnSymbol}
\usepackage[backref,colorlinks,citecolor=blue,bookmarks=true]{hyperref}
\usepackage[nameinlink]{cleveref}
\crefname{ineq}{inequality}{inequalities}
\usepackage{url}
\creflabelformat{ineq}{#2{\upshape(#1)}#3}
\usepackage[letterpaper,margin=1in]{geometry}
\usepackage{enumitem}

\newtheorem{theorem}{Theorem}[section]
\newtheorem{lemma}[theorem]{Lemma}

\newtheorem{definition}[theorem]{Definition}

\newtheorem{proposition}[theorem]{Proposition}

\newtheorem{conjecture}[theorem]{Conjecture}

\newcommand\codim{{\mathrm{codim}}}
\newcommand\SR{{\mathrm{SR}}}
\newcommand\End{{\mathrm{End}}}
\newcommand\Hom{{\textnormal{Hom}}}

\newcommand\J{\textnormal{J}}
\newcommand\SL{{\mbox{SL}}}

\newcommand\PSL{\textnormal{PSL}}
\newcommand\C{{\mathbb{C}}}

\newcommand\F{{\mathbb{F}}}

\newcommand\irr{{\mathop\textup{Irr}}}

\newcommand\remove[1]{{}}

\title{A Note on Slice Rank and Matchings in Groups}
\author{Kevin Pratt}

\begin{document}

\maketitle 
\begin{abstract}
A multiplicative 3-matching in a group $G$ is a triple of sets $\{a_i\}, \{b_i\}, \{c_i\} \subset G$ such that $a_ib_jc_k = 1$ if and only if $i=j=k$. Here we record the fact that $\PSL(2,p)$ has no multiplicative 3-matching of size greater than $O(p^{8/3})$, yet the slice rank of its group algebra's multiplication tensor is at least $\Omega(p^3)$ over any field. This gives a negative answer to a conjecture of Petrov.
\end{abstract}
\section{Introduction}
A \emph{multiplicative 3-matching} in a finite group $G$, hereon abbreviated to a 3-matching, is a triple of subsets $\{a_i\}_{i=1}^m, \{b_i\}_{i=1}^m, \{c_i\}_{i=1}^m$ of $G$ such that $a_ib_jc_k=1 \iff i=j=k$. Let $M(G)$ denote the largest size\footnote{By size we mean the parameter $m$.} of a 3-matching in $G$. This quantity is of interest in additive combinatorics, as finite groups provide a model setting for understanding 3-term arithmetic-progression-free sets in the integers. It also has connections to algorithms for fast matrix multiplication \cite{cohn2005group}.

The polynomial method of Croot, Lev, and Pach \cite{croot2017progression} and Ellenberg and Gijswijt \cite{ellenberg2017large}, and its formulation in terms of \emph{slice rank} due to Tao \cite{tao}, is a powerful tool for establishing upper bounds on $M(G)$. Remarkably, it gives an asymptotically tight bound on $M(G)$ in the case of $\F_p^n$ with $p>2$ a fixed prime. Specifically, it shows that for a certain $c_p < p$ we have $M(\F_p^n) \le c_p^n$, and it is also known that $M(\F_p^n) \ge c_p^{(1-o(1))n}$ \cite{kleinberg2016growth}. This raises the following question, previously asked in \cite{petrov2016combinatorial}: is the slice rank bound on $M(G)$ always tight?

We now make these notions formal. We assume throughout that $k$ is an algebraically closed field\footnote{This is without loss of generality for us, since slice rank can only decrease under field extensions and we will prove a slice rank lower bound.}. Following \cite[Lemma 1(iv)]{tao2}, the slice rank of a trilinear form $T : U_1 \times U_2 \times U_3 \to k$ is 
\[\SR(T) = \max_{\substack{V_i  \le U_i : \\ T(V_1,V_2,V_3) = 0}} \codim(V_1) +\codim(V_2) + \codim(V_3).\]
This can be thought of as an analogue of the ``codimension of the kernel" definition of matrix rank for trilinear forms. To use slice rank to prove bounds on $M(G)$, we consider the multiplication tensor  $T_{k[G]} : (k^{|G|})^{\times 3} \to k$ of the group algebra $k[G]$, defined as $T_{k[G]} = \sum_{g,h \in G} x_g y_h z_{gh}$. The key fact is that $\SR(T_{k[G]})$ is at least $M(G)$ for any $k$ \cite[Lemma~1]{tao}; this is an analogue of the fact that the rank of a matrix is at least the size of the largest identity-submatrix it contains. Hence upper bounds on $\SR(T_{k[G]})$ imply upper bounds on $M(G)$. Motivated by this, we make the following definition.

\begin{definition}\label{sr}
$\SR(G) = \min_k \SR(T_{k[G]})$.
\end{definition}

Here the minimum is taken over all algebraically closed fields; crucially, the characteristic of $k$ can be arbitrary. In the case that $G = \F_p^n$ for example, $\SR(T_{\C[\F_p^n]}) = p^n$ \cite[Corollary b.17]{blasiak2017groups} but $\SR(T_{\F_p[\F_p^n]}) \le c_p^n$. This leads to the following conjecture, which is slightly weaker than one appearing in \cite{petrov2016combinatorial}:\footnote{The conjecture of \cite{petrov2016combinatorial} asked if $M(G)$ is roughly the sum of codimensions of subspaces multiplying to 0 in $k[G]$, whereas \Cref{pconj} is equivalent to asking if $M(G)$ is roughly the sum of codimensions of subspaces whose product merely vanishes on the coefficient of 1 in $k[G]$.}

\begin{conjecture}\label{pconj}
$\SR(G) \le M(G) \cdot |G|^{o(1)}$. 
\end{conjecture}

In this work we note that $\PSL(2,p)$ has no large 3-matchings but has high slice rank over any field, so \Cref{pconj} is false. Both of these facts are in large part due to the lower bound of \cite{landazuri1974minimal} on the dimensions of nontrivial irreducible representations of $\PSL(2,p)$. The fact that $\PSL(2,p)$ has no large 3-matching follows almost immediately from Gowers's result on quasirandom groups \cite{gowers2008quasirandom}. We now give a quick proof of this. 

\begin{proposition}\label{easydir}
$M(\PSL(2,p)) \le O(p^{8/3})$.
\end{proposition}
\begin{proof}
A triple of subsets $A,B,C \subset G$ is called product-free if $abc \neq 1$ for all $a \in A, b \in B, c \in C$. If $A,B,C$ is a 3-matching of size $m$, then there is a product-free triple of sets in $G$ of size $m':=\lfloor m/3 \rfloor$ consisting of $\{a_i\}_{i=1}^{m'}, \{b_i\}_{i=m' + 1}^{2m'}, \{c_i\}_{i=2m'+1}^{3m'}$. In \cite{gowers2008quasirandom} it is shown that $\PSL(2,p)$ does not contain product-free subsets larger than $O(p^{8/3})$, so the proposition follows.
\end{proof}

In the next section we the following slice rank lower bound.

\begin{theorem}\label{main}
$\SR(\PSL(2,p)) \ge \Omega(p^3)$.
\end{theorem}


\section{Proof of \Cref{main}}
If $A$ is a finite dimensional algebra over $k$, we let $T_A \in A^* \otimes A^* \otimes A$ denote its multiplication tensor.  If $e_1, \ldots, e_n$ is a basis of $A$ with dual basis $e_1^*, \ldots, e_n^*$, this tensor is given in coordinates by $\sum_{1 \le i,j \le n} e_i^* \otimes e_j^* \otimes (e_i \cdot e_j)$. We can view this as a trilinear form, for instance by linearly mapping $e_i^* \otimes e_j^* \otimes e_k$ to the monomial $x_iy_jz_k$, and define its slice rank as in the introduction. We write $\SR(A)$ for the slice rank of the multiplication tensor of $A$.

Now we recall some basic facts about slice rank, namely that it is nonincreasing under linear transformations, and that the slice rank of an algebra is nonincreasing under quotients.

\begin{lemma}\cite[Lemma 3]{tao2}\label{gln}
Let $T = \sum c_{ijk} u_i \otimes v_j \otimes w_k \in U \otimes V \otimes W$, and let $A \in \Hom(U,U'), B \in \Hom(V,V'), C \in \Hom(W,W')$. Then $\SR(\sum c_{ijk} A(u_i) \otimes B(v_j) \otimes C(w_k)) \le \SR(T)$.
\end{lemma}

\begin{lemma}\label{quotients}
If $I$ is a two-sided ideal of $A$, then $\SR(A/I) \le \SR(A)$.
\end{lemma}
\begin{proof}
Let $\varphi : A \to A/I$ be the quotient map. Let $e_1, \ldots, e_n$ be a basis for $A$. Since $\varphi$ is an onto linear map, there exists $ S \subseteq [n]$ so that $\{\varphi(e_{i})\}_{i \in S}$ is a basis of $A/I$. Let $P : A \to A/I$ be given by $P(e_i) = \varphi(e_i)$ for $i \in S$, and $P(e_i) = 0$ if $i \notin S$. Similarly define $P' : A^* \to (A/I)^*$ by $P'(e_i^*) = P(e_i)^*$. Applying $P'$ to the first two factors of $T_A$ and $P$ to the third, we obtain $T_{A/I}$. By \Cref{gln} this proves the claim.
\end{proof}
For an algebra $A$, we denote by $\irr(A)$ the set of non-isomorphic irreducible representations of $A$ over $k$. Recall that in the ordinary (i.e.,~characteristic 0) representation theory of finite groups, representations are completely reducible, and in particular the group algebra $k[G]$ is isomorphic to a direct sum of matrix algebras. While this is false when the characteristic of the field divides the order of the group, we still have the following. 

\begin{definition}
The radical of $A$, denoted $\J(A)$, is the two-sided ideal of all elements of $A$ which act by 0 on all irreducible representations of $A$.
\end{definition}

\begin{lemma}\cite[Theorem 2.12]{etingof2011introduction}\label{keylem}
Let $A$ be a finite-dimensional algebra. Then 
\[A/\J(A) \cong \bigoplus_{V \in \irr(A)} \End(V).\]
\end{lemma}

We will use the following fact, which says that the slice rank of direct sums of matrix multiplication tensors is maximal.
\begin{lemma}\label{mmsr}\cite[Proposition B.6]{blasiak2017groups}
For any field $k$, $\SR(\bigoplus_{i=1}^m \End(k^{d_i})) = \sum_{i=1}^m d_i^2$.
\end{lemma}

The proof of \Cref{main} will go as follows. By \Cref{keylem}, $k[G]/\J(k[G])$ is a direct sum of matrix algebras, one for each irreducible representation of $k[G]$. So by \Cref{mmsr}, $k[G]/\J(k[G])$ has full slice rank, and by \Cref{quotients} this is a lower bound on the slice rank of $k[G]$. So if we can show that there are many sufficiently large irreps of $G$, we conclude that $\SR(k[G])$ is large. To give an example of when this fails dramatically, when $G$ is any $p$-group, the only irrep of $G$ when $k$ has characteristic $p$ is the trivial one \cite[Corollary of Proposition 26]{serre1977linear}, so this argument only says that $\SR(k[G]) \ge 1$.

\begin{proof}[Proof of \Cref{main}]
Let $G = \PSL(2,p)$ and let $k$ be a field of characteristic $\ell$. First, if $\ell=0$ or $\ell$ is coprime to $|G| = (p-1)p(p+1)/2$, then $k[G]$ is semisimple and so by \Cref{mmsr} the slice rank of $k[G]$ equals $|G| = \Omega(p^3)$. Next, if $\ell = p$, then the irreps of $\SL(2,p)$ are given by the action of $G$ on homogeneous polynomials in two variables of degree up to $p-1$  with coefficients in $k$ \cite[p.~15]{alperin1993local}; since the center of $\SL(2,p)$ acts trivially on even degree polynomials, these are also irreps of $\PSL(2,p)$ (in fact, all of them). By \Cref{keylem} the dimension of $k[G]/\J(k[G])$ is then  $\sum_{i=0}^{(p-1)/2} (2i+1)^2 \ge \Omega(p^3)$, so the claim holds.

So suppose $\ell \neq p$ divides $|G| = (p-1)p(p+1)/2$. By \cite[I.3 Theorem 2]{alperin1993local}, the number of irreducible representations of $G$ equals the number of conjugacy classes having order coprime to $\ell$. Next we show that there are $\Omega(p)$ such conjugacy classes of $G$. This follows from the more general bound of \cite[Theorem 6.1]{hung2022p}; here we sketch a proof for the special case of $\PSL(2,p)$. See \cite[p.~71]{fulton2013representation} for a reference on conjugacy classes of $\SL(2,p)$, which we adapt to $\PSL(2,p)$. Most elements in $G$ are either conjugate to an element of the \emph{split torus}, a cyclic subgroup of order $(p-1)/2$, or the \emph{non-split torus}, a cyclic subgroup of order $(p+1)/2$. The number of non-conjugate elements in the split torus is at least $(p-3)/4$, and the number of non-conjugate elements in the non-split torus is at least $(p-5)/4$ (with the exact values depending on $p \bmod 4$). Because the orders of these tori are coprime, all conjugacy classes of elements in at least one of the subgroups have order coprime to $\ell$. So there are at least $(p-5)/4$ such conjugacy classes. 

Finally, since the minimum dimension of a nontrivial irrep of $G$ is at least $(p-1)/2$ in characteristic $\ell \neq p$ \cite{landazuri1974minimal}, we conclude by \Cref{keylem} that $\dim k[G]/\J(k[G]) \ge \Omega(p^3)$, and thus by \Cref{mmsr} $\SR(k[G]) \ge \Omega(p^3)$.
\end{proof}

One might wonder if all sufficiently quasirandom groups (groups with no small nontrivial irreps over $\C$) have high slice rank. Here is a conjecture towards this question.
\begin{conjecture}
For a fixed $\varepsilon > 0$, let $G$ be a group of order $n$ that is $n^\varepsilon$-quasirandom. Then for all fields $k$, we have the uniform bound of $\dim k[G]/\J(k[G])\ge \Omega(n)$.
\end{conjecture}

\section{Acknowledgments}
I thank Ryan O'Donnell for feedback on a previous draft, and Josh Grochow for helpful clarifications.
\bibliographystyle{amsalpha}
        \bibliography{refs}
\end{document}